\documentclass[12pt, a4 paper]{amsart} 

\usepackage[psamsfonts]{amssymb} 
\usepackage{amsfonts,amsmath,mathabx}
\usepackage{color}

\usepackage{amsthm, amssymb}
\usepackage{amsfonts} 
\usepackage{epsfig,multicol} 

\usepackage{xypic}
\usepackage{hyperref}
\usepackage{psfrag}

\input xy
\xyoption{all}

\makeatletter
\newtheorem*{rep@theorem}{\rep@title}
\newcommand{\newreptheorem}[2]{%
\newenvironment{rep#1}[1]{%
 \def\rep@title{#2 \ref{##1}}%
 \begin{rep@theorem}}%
 {\end{rep@theorem}}}
\makeatother

\newtheorem{thmA}{Theorem}

\newtheorem{theorem}{Theorem}[section]

\newtheorem{proposition}[theorem]{Proposition}
\newtheorem{lemma}[theorem]{Lemma}

\newreptheorem{theorem}{Theorem}
\newreptheorem{proposition}{Proposition}
\newreptheorem{lemma}{Lemma}

\theoremstyle{remark}
\newtheorem{remark}[theorem]{Remark}
\newtheorem{remarks}[theorem]{Remarks}
\newtheorem{example}[theorem]{Example}
\newtheorem{notation}[theorem]{Notation}

\theoremstyle{definition}
\newtheorem{definition}[theorem]{Definition}

\def\Z{\mathbb Z}
\def\N{\mathbb N}

\def\R{\mathbb R}
\def\To{\mathbb T}
\def\G{\Gamma}

\def\S{\mathbb S}

\def\-{\overline}
\def\wh{\widehat}

\def\rk{{\rm{rk}}}

\def\epi{{\rm{Epi}}}

\def\x{\underline{x}}

\def\T{\mathcal{T}}
\def\J{\mathcal{J}}
\def\P{\mathcal{P}}

\def\G{\Gamma}

\def\<{\langle}
\def\>{\rangle}

\def\x{\underline{\bf{x}}}
\def\y{\underline{\bf{y}}}
\def\z{\underline{\bf{z}}}

\begin{document}

\title[Profinite isomorphism]{The isomorphism problem for profinite completions of residually finite groups}

% author  information
\author[Martin R. Bridson]{Martin R.~Bridson}
\address{Martin R.~Bridson, Mathematical Institute, University of Oxford, Andrew Wiles Building, Oxford OX2 6GG, UK}
\email{bridson@maths.ox.ac.uk} 
 
\author[Henry Wilton]{Henry Wilton}
\address{Henry Wilton, Department of Mathematics, University College London, Gower Street, London WC1E  6BT, UK}
\email{hwilton@math.ucl.ac.uk}

\date{10 January 2014}

\thanks{Both authors are supported by the EPSRC. Bridson is also supported by a Wolfson Research Merit Award from the Royal Society.}

\begin{abstract}
We consider pairs of finitely presented, residually finite groups $u:P\hookrightarrow
\G$. We prove that there is no algorithm that, given an arbitrary such
pair, can determine whether or not the associated map of profinite completions $\hat{u}: \wh{P}
\to \wh{\G}$ is an isomorphism. Nor do there exist algorithms that can decide whether
$\hat{u}$ is surjective, or whether $ \wh{P}$
is  isomorphic to $ \wh{\G}$.
\end{abstract}

\maketitle

\centerline{\small{\em{For Pierre de la Harpe on his 70th birthday, with respect and affection}}}

\def\G{\Gamma}
\def\epi{{\rm{Epi}}}

\section{Introduction} The profinite completion of a group $\G$ is the inverse
limit of the directed system of
finite quotients of $\G$; it is denoted $\wh{\G}$. 
The natural
map $\G\to \wh{\G}$ is injective if and only if $\G$ is residually finite.
A  \emph{Grothendieck pair} is a monomorphism $u: P\hookrightarrow\Gamma$ 
of finitely presented, residually finite groups such
that $\hat u :\wh P\to\wh\Gamma$ is an isomorphism but $P$ is not isomorphic to 
$\Gamma$. The existence of Grothendieck pairs was established in \cite{BG}, 
raising the following recognition problem:
{\em is there an algorithm
that, given a monomorphism of finitely presented, residually finite groups $u: P\hookrightarrow\Gamma$,
can determine whether or not $\hat{u}$ is an isomorphism? }

We shall resolve this question by proving the following theorem.

\begin{thmA}\label{t:main}
There are recursive sequences of finite presentations for residually finite groups 
$P_n=\<A_n \mid R_n\>$ and $\G_n=\<B_n\mid S_n\>$ together with explicit monomorphisms $u_n:P_n\hookrightarrow \G_n$,
 such that:
\begin{enumerate}
\item $\widehat{P}_n\cong\widehat{\G}_n$ if and only if the induced map of profinite completions $\hat{u}_n$ is an isomorphism; 
\item $\hat{u}_n$ is an isomorphism if and only if $\hat{u}_n$ is surjective; and
\item the set $\{n\in\N\mid \widehat{P}_n\not\cong\widehat{\G}_n\}$
is recursively enumerable but not recursive.
\end{enumerate}
\end{thmA}

The groups $\G_n$ that we construct are of the form $H_n\times H_n$,
where $H_n$ is a residually finite hyperbolic group with a 2-dimensional classifying space.

Theorem \ref{t:main} is an outgrowth of our recent work on the
profinite triviality problem for finitely presented groups \cite{BWprotrivial}:
we shall prove it by applying
 the main construction of \cite{BG} to a sequence of finitely
presented groups satisfying the following strengthening of  \cite[Theorem A]{BWprotrivial}.

\begin{thmA}\label{t:tech}
There is a recursive sequence of finite combinatorial\footnote{A {\em combinatorial}
map between complexes is one that maps open cells homeomorphically to open
cells, and a CW complex is {\em combinatorial} if the attaching map of each closed $k$-cell is a
combinatorial map from a polyhedral subdivision of the $(k-1)$-sphere.} CW-complexes $K_n$ 
%(with a uniform bound on the number of cells in each) 
so that
\begin{enumerate}
\item each $K_n$ is aspherical;
\item $H_1(K_n,\Z)\cong H_2(K_n,\Z)\cong 0$ for all $n\in\N$; and
\item  the set of natural numbers
\[
\{n\in\N\mid \widehat{\pi_1K_n}\ncong 1\}
\]
is recursively enumerable but not recursive.
\end{enumerate}
\end{thmA}

In \cite[Theorem B]{mb:H2} the first author proved a less satisfactory version of Theorem \ref{t:main}.
 He too constructed a recursive sequence of pairs $\iota_n:P_n\hookrightarrow\G_n$,
with $P_n$ and $\G_n$ residually finite, such that there is no algorithm that can determine whether or not  $\hat{\iota}_n$ is an
isomorphism, nor whether $\wh{P}_n$ is isomorphic to $\wh{\G}_n$. But in his construction, although
the $\G_n$ are given by explicit
finite presentations, the subgroups  $P_n\hookrightarrow\G$ are given by specifying a finite generating set $\Sigma_n\subset\G$
(with a guarantee that $P_n=\<\Sigma_n\>$ is  finitely presentable). In \cite{BWggd} we explored in detail the
question of when such data is sufficient to allow the algorithmic construction of a finite presentation for $P_n$,
and in this situation it is not.  Indeed, the lack of an algorithm to present $P_n$ is an essential feature of
the proof  of  \cite[Theorem B]{mb:H2}: 
what is actually  proved is that one cannot decide if 
$\hat{\iota}_n$ is {\em injective} (note the contrast with our Theorem \ref{t:main})
 because there is no algorithm that can determine if the map $H_1(P,\Z)\to H_1(\G,\Z)$
induced by $\iota_n$ is injective; if one had a finite presentation for both groups, 
it would be easy to determine if $H_1(P,\Z)\to H_1(\G,\Z)$ were injective.

Theorem A deals with maps 
$u:P\rightarrow \G$ that are assumed to be {\em{monomorphisms}} of {\em{residually finite}} groups. If one
does not require the groups to be residually finite, then Theorem A of \cite{BWprotrivial} 
implies immediately that there does not exist  an algorithm to recognise if $\hat{u}:\wh{P}\to\wh{\G}$ is an isomorphism: it is enough to consider the case where $P$ is the trivial subgroup. If one also drops the requirement that $u:P\to\G$ is injective, then the result already follows from the work of Slobodskoi  that is invoked in \cite{BWprotrivial}:  Slobodskoi \cite{Slobo} exhibited a finitely presented group $G$ for which there is no algorithm that can determine which words $w$ have trivial image in $\wh{G}$; and since finitely generated profinite groups are Hopfian, $w$ has trivial image in $\wh{G}$ if and only if  the quotient map $G\to G /\langle\!\langle w\rangle\!\rangle$ induces an isomorphism of profinite completions.

\section{Reducing Theorem \ref{t:main} to Theorem \ref{t:tech}}

The reduction of Theorem \ref{t:main} to Theorem \ref{t:tech} follows a template for constructing 
Grothendieck pairs that was devised by Bridson and Grune-wald in \cite{BG} and first applied
to algorithmic questions in \cite{mb:karl}. 
The key ingredients are adaptations of the Rips construction \cite{rips} and the 1-2-3 Theorem \cite{bbms}.
 In our case, these will be applied to the
groups $\pi_1K_n$ supplied by Theorem \ref{t:tech}. 

\subsection{Rips and Algorithmic 1-2-3}

The following adaptation of the Rips construction is due to Wise \cite{wise:rf}.

\begin{theorem}\cite{wise:rf}\label{t:rips} There is an algorithm that takes as input a finite group presentation $\mathcal Q\equiv\<\x\mid \y\>$
and will output a finite presentation $\mathcal H\equiv\<\x,a,b,c\mid \z\>$ for a residually finite hyperbolic group $H$, so that
there is a short exact sequence
$$
1\to N\to  H\overset{p}\to Q \to 1,
$$
where $N=\<a,b,c\>$, the group $Q$ is presented by $\mathcal Q$, and $p:H\to Q$ is the map
implicit in the labelling of generators. Moreover, $|\z| = |\y| + 6|\x|$.
\end{theorem}

The short exact sequences obtained by applying this theorem with $Q=\pi_1K_n$ provide input for the following
algorithmic version  of the 1-2-3 Theorem.  Recall that the \emph{fibre product} associated to a short exact sequence
\[
1\to N\to  H\overset{p}\to Q \to 1
\]
is the pre-image of the diagonal subgroup under the product map $p\times p:H\times H\to Q\times Q$.

\begin{theorem}\cite{bhms}\label{t:123}
 There exists an algorithm that, given a short exact sequence of groups
$1\to N\to H\to Q\to 1$, a finite generating set for $N$, a finite presentation of $H$,
and a finite combinatorial CW-complex that is the 3-skeleton of 
a $K(Q,1)$, will construct a finite presentation for the
associated fibre product $P<H\times H$, and express the generators of $P$ as words in
the generators of $H\times H$.
\end{theorem}

The theorem proved in \cite{bhms} is actually somewhat stronger than the one we have stated
and is cast in different language; we shall explain the translation. The algorithm in
\cite{bhms}  requires as input  a finite presentation for $H$, a finite generating set for $N$, a finite
presentation $\mathcal Q$ for $Q$, and a set of elements of $\pi_2\mathcal Q_n$
that generate it as a $\Z Q_n$-module. (Our results in \cite{BWggd} show that this last piece
of data is essential.) 

In \cite{bhms}, the generators of $\pi_2\mathcal Q$ are given as identity sequences, which 
can be thought of as formal products $\Theta\equiv \Pi_{i=1}^N x_ir_ix^{-1}$, where the $x_i$ are
words in the generators of $\mathcal Q$, the $r_i^{\pm 1}$ are defining relations in $\mathcal Q$,
and $\Theta$ is freely equal to the empty word. Given a  combinatorial map of a 2-sphere to the
standard 2-complex $K(\mathcal{Q})$ of the presentation $\mathcal Q$, 
it is easy to associate such a product to it: this is 
the first, simple step in the proof of van Kampen's Lemma where, given a diagram with boundary
word $w$, one cuts it into a ``lollipop diagram", the boundary of which is labelled by a
product of conjugates of defining relations that is freely equal to $w$; we are 
discussing the case $w=\emptyset$ (see \cite{LS} p.151 or \cite{mrb:bfs} p.49, for example).
Thus, instead of giving the generators of $\pi_2\mathcal Q$ as identity sequences, it suffices
to give them as combinatorial maps $\Bbb S^2\to K(\mathcal Q)$. 

If one attaches to $K(\mathcal Q)$
one 3-cell  for each element in a set of  maps $\Bbb S^2\to K(\mathcal Q)$, then the resulting 3-complex
will be the 3-skeleton of a $K(Q,1)$ if and only if the homotopy classes of these maps 
generate $\pi_2K(\mathcal Q)$ as a $\Z Q$-module. Changing perspective, if one is given 
a finite combinatorial model for the
3-skeleton of a $K(Q,1)$, one can shrink a maximal tree in the 1-skeleton to make the 2-skeleton
a presentation 2-complex, and the set of attaching maps of the 3-cells will serve as a set of
generators for $\pi_2$ of this presentation.

\subsection{Criteria for profinite equivalence}

Together, Theorems \ref{t:rips} and  \ref{t:123}  provide a mechanism that, given a
finite combinatorial complex with 2-connected universal cover and fundamental
group $Q$, will output finite presentations for $H\times H$ and $P$, and 
will describe the inclusion $u: P\hookrightarrow H\times H$ by expressing the generators of $P$ as words in the
generators of $H\times H$. The final step in the reduction of Theorem \ref{t:main} to Theorem \ref{t:tech}
involves identifying when $\hat{u}$ is surjective, when $\hat{u}$ is an isomorphism, and
when $\wh{P}\cong \wh{H\times H}$. 

\begin{lemma}\label{l:uEpi} Let $1\to N\to H\to Q\to 1$ be an exact sequence of groups
and let $P<H\times H$ be the associated fibre product. 
The maps that $N\hookrightarrow H$
and $P\hookrightarrow H\times H$ induce on profinite completions are surjective if and only if
$\wh{Q}=1$.
\end{lemma}

\begin{proof} 
If $Q$ has no finite quotients then $N$ is dense in $\wh{H}$, so
$\wh{N}\to \wh{H}$, which has compact image, is surjective.
Likewise,  
$N\times N< P$ is dense in $\wh{H\times H}$, so $\wh{P}\to \wh{H\times H}$ is surjective.

Conversely, if there is a surjection from
$Q$ to a non-trivial finite group $F$, then the composition $H\to Q\to F$ extends to $\wh{H}$
and the kernel of this extension is a proper closed subgroup containing $N$.
And the image of $P$ in $\wh{H\times H}$ is contained in the
proper closed subgroup obtained by pulling back the diagonal of $F\times F$.
\end{proof}

The question of when an inclusion induces an isomorphism of profinite completions is more subtle. The following criterion, due to Platonov and Tavgen \cite{PT}, plays a central role in \cite{PT}, \cite{BG} and \cite{mb:karl} --- see \cite[Theorem 5.1]{BG} for a proof.

\begin{proposition}\label{p:PT} Let $1\to N\to H\to Q\to 1$ be an exact sequence of finitely
generated groups with fibre product $P$. If $H_2(Q,\Z)=0$ and $\wh{Q}=1$ then the inclusion
$u:P\hookrightarrow H\times H$ induces an isomorphism of profinite completions.
\end{proposition}

We also need to be sure that $\wh{P}\not\cong\wh{H\times H}$ when $\hat{u}$ is not an isomorphism. 
Our proof of this requires a lemma and some notation.

We write $\epi(G,S)$ for the set of epimorphisms from a group $G$ to a finite group $S$.
Composition of epimorphisms with the natural map $G\to \wh{G}$ defines a
bijection $\epi(\wh{G}, S)\to \epi(G,S)$, so  if $\wh{G}_1\cong\wh{G}_2$ then
there is a bijection $\epi(G_1, S)\to \epi(G_2,S)$

\begin{lemma}\label{lem: Epi equation}
Let $G$ be a finitely generated group, let $N_1, N_2 <G$
be normal subgroups that commute, let $\G_i=G/N_i$, let $Q=G/N_1N_2$,
and let $S$ be a non-abelian finite simple group. Then 
$$
|\epi(G,S)| = |\epi(\G_1, S)| + |\epi(\G_2,S)| - |\epi(Q, S)|.
$$
\end{lemma}

\begin{proof} Since $S$ is simple, any epimorphism $\phi: G\to S$ must either
restrict to an epimorphism on $N_i$ or else map $N_i$ trivially. Since
$S$ is not abelian and $N_1$ commutes with $N_2$, it cannot be that $\phi(N_1)=\phi(N_2)=S$.
Thus $\epi(G,S)$ is the union of the set of epimorphisms 
that map $N_1$  trivially and the set of those that map $N_2$ trivially; the former is
bijective with $\epi(\G_1,S)$ and the latter with $\epi(\G_2,S)$. The intersection of these
sets is the set of epimorphisms where $N_1N_2$ maps trivially, and this is bijective with
$|\epi(Q, S)|$. 
\end{proof}

\begin{proposition}\label{p:epic} Let $P<H\times H$ be the fibre product of $1\to N\to H\to Q\to 1$,
where $H$ is finitely generated and $Q$ is finitely presented.
Let $S$ be a non-abelian finite simple group $S$. There is a bijection from
$\epi(P, S)$ to $\epi(H\times H,S)$ if and only if $\epi(Q,S)=\emptyset$.
\end{proposition}

\begin{proof} The 0-1-2 Lemma (see, for instance, \cite[Lemma 2.1]{BG}) tells us that $P$ is finitely generated, so all of the
sets $\epi(\cdot, S)$ under consideration are finite. Let $N_1=P\cap (H\times 1)$,
let $N_2=P\cap (1\times H)$, and note that $P/N_i\cong H$ (this is the image of $P$ under
projection to one of the factors of $H\times H$). Also, $P/N_1N_2\cong Q$.
Lemma \ref{lem: Epi equation} tells us that $|\epi(P,S)| = 2|\epi(H,S)|$ if and only if $\epi(Q,S)=\emptyset$.  A second application of that lemma, with $H\times 1$ in the role of $N_1$
and $1\times H$ in the role of $N_2$, and $H\times H$ in the role of $G$, implies that $|\epi(H\times H, S)|= 2|\epi(H,S)|$. 
\end{proof}

\subsection{Summary of the reduction}

Let $K_n$ be the sequence of complexes given by Theorem \ref{t:tech} and let $Q_n=\pi_1K_n$.  
Combining Theorems \ref{t:123} and \ref{t:rips}, we obtain (by means of an algorithmic procedure) finite presentations for residually finite groups $P_n$ and $\G_n := H_n\times H_n$ and a monomorphism $u:P_n\hookrightarrow\G_n$.  Lemma \ref{l:uEpi} tells us that $\hat{u}$ is an epimorphism if and only if $\wh{Q}_n\cong 1$, and Proposition \ref{p:PT} tells us that $\wh{Q}_n\cong 1$ also implies that $\hat{u}$ is an isomorphism. Finally, since $Q_n$ is perfect,  $\wh{Q}_n\ncong 1$ implies $\epi(Q,S)\neq\emptyset$ for some finite non-abelian simple group $S$, in which case $\wh{P}_n\not\cong \wh{\G}_n$, by Proposition \ref{p:epic}.
\qed

\section{Universal Central Extensions}

In our proof of Theorem \ref{t:tech} we will need to construct and present universal central extensions
of groups.

A {\em central extension} of a group $Q$ is
a group $\widetilde Q$ equipped with a homomorphism $\pi:\widetilde{Q}\to Q$ whose kernel is central
in $\widetilde Q$. Such an extension is {\em universal} if given any other central extension
$\pi': E \to Q$ of $Q$, there is a unique homomorphism $f : \widetilde Q \to  E$ such that
$\pi'\circ f = \pi$. 

The standard reference for universal central extensions is 
\cite{milnor} pp. 43--47. The
properties that we need here are these: $Q$ has a universal central extension $\widetilde Q$
if (and only if) $H_1(Q, \Z) = 0$; there is a short exact sequence
$$1 \to  H_2(Q,\Z) \to \widetilde{Q}\to Q\to 1\ ;$$
and if $Q$ has no non-trivial finite quotients, then neither does $\widetilde Q$.

The following result is Corollary 3.6 of \cite{mb:karl}; the proof relies on an argument
due to Chuck Miller.

%% clash of notation removed

\begin{proposition}\label{p:tilde}
 There is an algorithm that, given a finite presentation $\<A\mid {R}\>$
of a perfect group $Q$, will output a finite presentation $\< A\mid \widetilde{R}\>$ for the
universal central extension $\widetilde Q$. Furthermore, $|\widetilde{R}| = |A|(1 + |{R}|)$.
\end{proposition}

\section{Classifying spaces for universal central extensions}

In our earlier work on the triviality problem for profinite completions, we considered a construction that transforms a presentation $\P$ for a group $G$ into an aspherical complex $\mathcal{J}(\P)$ with fundamental group $\J_G$, such that $\widehat{\J}_G=\widehat{G}$.  (See \cite[Definition 7.1]{BWprotrivial}, in which $\mathcal{J}(\P)$ is denoted by $S(\mathcal{\P})$.)  In this section, our main task is to algorithmically construct an aspherical complex that represents the universal central extension of $\J_G$.  We start by recalling the construction of $\mathcal{J}(\P)$.

\subsection{The main construction}\label{ss:main}

Let $X$ be the standard 2-complex of a finitely presented group $J=\<B\mid S\>$  and let $c$ be a parameterization of the edge corresponding to a generator $\alpha$ of infinite order.  Let $\P \equiv \<a_1,\dots,a_n \mid r_1,\dots, r_m\>$ be a presentation of a group $G$.
%Suppose that $H_1(X,\Z)\cong H_2(X,\Z)\cong 0$, 

\begin{definition}
We define $\mathcal J(\P)$ to be the 2-complex obtained from the standard 2-complex  of the presentation $\P$ by replacing each 2-cell with a copy of $X$ that is attached with an annulus along which one can homotop the boundary cycle of the deleted 2-cell to the loop $c$ in $X$.   Let $\mathcal J_G$ be the fundamental group of $\mathcal J(\P)$.
\end{definition}

Using subscripts to distinguish generators and relations in the $m$ disjoint copies of $J$, we get the
following presentation of $\mathcal J_G$:
$$
\<B_1,\dots, B_m, a_1,\dots, a_n \mid S_1, \dots, S_m, r_1\alpha_1^{-1},\dots, r_m\alpha_m^{-1}\>~.
$$
\begin{remark}\label{r:rem}
\begin{enumerate}
\item The construction of $\mathcal{J}(\P)$ from $\P$ is entirely algorithmic.
\item The group $\J_G$ is the fundamental group of a tree of groups $\mathcal T$; the underlying tree is $T_m$, as defined
in Definition \ref{d:complexes}, the local group at the central vertex is
free on $\{a_1,\dots,a_n\}$, the other vertex groups are copies of $J$, and the edge groups are infinite cyclic.
\item There is a canonical map $\J_G\to G$ obtained by quotienting out the generators $B_i$, for all $i$.
\end{enumerate}
\end{remark}

We shall require the following basic properties of this construction. 

\begin{proposition}\label{p:properties}
Let $X$, $J$, $\P$ and $G$ be as above.
\begin{enumerate}
\item If $J$ is perfect, then $H_1(\mathcal J_G,\Z)\cong H_1(G,\Z)$.
\item If $X$ is aspherical, then $\mathcal J(\P)$ is aspherical. 
\item If $\wh{J}\cong 1$, then the canonical map $\J_G\to G$ induces an isomorphism $\wh{\J}_{G}\overset{\cong}\to\widehat{G}$.
\end{enumerate}
\end{proposition}
\begin{proof}
If $J$ is perfect, then in the abelianisation of $\mathcal J_G$, the generators from $B_i$ must all have trivial image for $i=1,\dots,m$. It follows that the abelianisation of $\mathcal J_G$ is the abelianisation of the group obtained by deleting all occurrences of the letters $B_i$ from the given presentation of  $\mathcal J_G$, and this group is $G$.  This proves (1).

Assertion (2) is a standard fact about graphs of aspherical spaces. See, for instance \cite[Proposition 3.6(ii)]{SW}.

To prove (3), note that since $J_i\cong J$ for $i=1,\dots,m$ and
$\wh{J}\cong 1$, each  $J_i<\J_G$ has trivial
image under every homomorphism $\phi:\J_G\to Q$ to a finite group. Thus $\phi$
factors through the quotient of $\J_G$ by the normal closure of the $J_i$, as required.
\end{proof}

\subsection{Tubular bundles}

We now introduce a class of spaces which conveniently describe the universal central extensions of groups of the form $\J_G$.

\begin{definition}\label{d:complexes} Let $d$, $n$ and $m$ be positive integers and let $T_m$ be the simplicial
tree with vertex set $\{v_0,\dots,v_m\}$ and edge set $\{ \{v_0,v_i\} : i=1,\dots m\}$. 
A {\em{tubular bundle of type $(d;n,m)$}}
is defined by the following data:
\begin{enumerate} 
\item a family of finite combinatorial 2-complexes $\underline{V}=( V_i \mid i=1,\dots,m)$;
\item a family $\underline{C}=( c_i \mid i=1,\dots,m)$, where
$c_i:\S^1\to V_i^{(1)}$ is a locally injective loop in the 1-skeleton of $V_i$;
\item a family of non-trivial reduced words $\underline{R}=(\rho(i) \mid i=1,\dots,m)$ in the free group $F_n
={\rm{Free}}(a_1,\dots,a_n)$ of rank $n$;
\item a family of elements $\underline{Z}=(z_i\in \Z^d \mid i=1,\dots, m)$. 
\end{enumerate}
Canonically associated to these data, there is a $(d+2)$-dimensional space $K(\underline{V}, \underline{C},
\underline{R}, \underline{Z})$ defined as follows.

Let $Y_n$ denote the $n$-rose (1-vertex graph) whose 1-cells are oriented and labelled 
$a_1,\dots,a_n$,
and identify
words $u\in F_n$ with  edge-loops $\gamma_u:\S^1\to Y_n$ based at the vertex.  Let $\To=\R^d/\Z^d$ be the standard $d$-torus. 
Let $\theta:[0,1]\to\S^1$ be
a parametrisation of $\S^1$.
\iffalse
and 
for each $z\in Z^d$ define
$\phi_z : \To\times \S^1 \to \T\times \S^1$ to be the unique basepoint-preserving affine map
that is the identity on $\T\times\{0}$ and sends the homotopy class $[\theta]$ to $[\theta]+z$. 
\fi
Then
$K(\underline{V}, \underline{C},
\underline{R}, \underline{Z})$ is the quotient of
$$
(\To\times Y_n) \sqcup \coprod_{i=1}^m \left(\To\times \S^1 \times [0,1]\right)_i\sqcup \coprod_{i=1}^m (\To \times V_i)$$
by the equivalence relation that attaches $(\To\times \S^1 \times\{0\})_i$ to $\To\times Y_n$ by
$(t,\theta,0)\sim (t, \gamma_{\rho(i)}(\theta))$ and attaches $(\To\times \S^1 \times\{1\})_i$ to $\To\times V_i$ by  $(t,\theta,1)\sim (t + \theta z_i,c_i(\theta))$.
\end{definition}

\begin{notation} 
When $V_i$ and $c_i$ are independent of $i$ (equal to $X$ and $c$, say) we adopt the abbreviated notation 
$K^{X, c} (\underline{R}, \underline{Z})$ and say that the complex is a 
{\em{tubular bundle of type $(d;n,m)$ over $(X,c)$}}. 

A sequence of such complexes $K_n = K^{X, c} (\underline{R}_n, \underline{Z}_n)$ is termed {\em{recursive}}
if the sequences $\underline{R}_n$ and $ \underline{Z}_n$ are recursively enumerable.
\end{notation}

\begin{example}\label{l:all}
For fixed $(X,c)$ and $(d;n,m)$, the set of all tubular bundles of type $(d;n,m)$ over $(X,c)$ is recursive.
\end{example}

\begin{remark} 
(1) Because the attaching maps of the edge spaces
%$\To\times\S^1\to \To\times\S^1$ defined by $(t,\theta,1)\sim (t + \theta z_i, \theta)$
are affine, there is a simple algorithmic process that will subdivide 
$K^{X, c} (\underline{R}, \underline{Z})$  into a combinatorial CW complex, thence a simplicial complex, but
it is more natural to work with tubular bundles since they faithfully and naturally encode the geometry of the
situation that interests us, with an economical amount of defining data.

\smallskip

(2) A technical advantage of tubular bundles over combinatorial CW-complexes is the flexibility afforded by the
fact that the
attaching maps of the edge spaces can map $k$-cells across cells of dimension greater than $k$.
This is equally natural and 
desirable when trying to cellulate torus bundles over the circle, for example.

\smallskip

(3) In Definition \ref{d:complexes} we restricted to the case where the underlying graph is a tree, the fibre is a
torus, and the vertex spaces are combinatorial 2-complexes. We did so because this is the case that is needed in the
proof of Theorem \ref{t:tech}, but it is easy to imagine other contexts in which one would want to relax
these constraints.
\end{remark}

Our interest in tubular bundles derives from the following proposition, which is expressed in the notation of
Subsection \ref{ss:main}, in particular $G$ is the group with presentation
 $\<a_1,\dots,a_n \mid r_1,\dots, r_m\>$ and $X$ is the standard 2-complex of a presentation for the group
 $J$, with $\alpha\in J$ a generator of infinite order.

\begin{proposition}\label{p: tubular properties} %%mrb
If $X$ is aspherical, $G$ is perfect and  $H_1(J,\Z)\cong H_2(J,\Z)\cong 0$, then the universal central extension of $\mathcal J_G$ is the fundamental group of an aspherical tubular bundle $K^{X, c} (\underline{R}, \underline{Z})$ of type $(d;n,m)$, where $d=m-n$ and $\underline{R}=(r_i \mid i=1,\dots,m)$.
 
\end{proposition}
\begin{proof}
Proposition \ref{p:properties} assures us that $\mathcal J_G$ is perfect and therefore has a universal central extension. It also assures us that
the 2-dimensional complex
$\mathcal J(\P)$ is aspherical, so $H_2(\mathcal J_G,\Z)$ is the kernel of the second boundary map in the cellular chain complex of $\mathcal J(\P)$,
hence free abelian. The rank of this group can be calculated by Euler characteristic: on the one hand, since $H_1(\mathcal J_G,\Z)\cong 0$, we have $\chi(\mathcal{J}(\P)) = \rk~H_2(\mathcal J(\P),\Z) +1$. On the other hand, counting cells, we have $\chi(\mathcal J(\P))=m\chi(X) + (1-n)$, whence $\chi(\mathcal{J}(\P))=m-n+1$, since $H_1(X,\Z)\cong H_2(X,\Z)\cong0$ implies that $\chi(X)=1$.

Thus the universal central extension of $\mathcal J_G$ has the form
\begin{equation}\label{ucc}
1\to \Z^{m-n} \to \widetilde{\mathcal{J}}_G\overset{p}\to \mathcal{J}_G\to 1.
\end{equation}
We are assuming that $H_2(J,\Z)= H_2(X,\Z)\cong 0$, so restricting to $J_i<G$ for each $i=1,\dots,m$, we get 
a central extension 
$$1\to \Z^{m-n} \to p^{-1}(J_i) \overset{p}\to J_i\to 1$$ that splits. And  
$1\to \Z^{m-n} \to p^{-1}(F_n) \overset{p}\to F_n\to 1$ also splits, where $F_n=\pi_1Y_n$.
We fix splittings $\sigma_i:J_i\to \mathcal J_G$ and $\sigma_0:F_n\to \mathcal J_G$. Since (\ref{ucc}) itself does not split,
the  splittings $\sigma_i$ and $\sigma_0$ cannot all agree on the edge groups $\<\alpha_i\> = J_i\cap F_n$, so at least one
of the  elements $z_i:=\sigma_0(r_i)^{-1} \sigma_i(\alpha_i)\in \Z^{m-n}$ is non-trivial. 

%%mrb
Via the projection $\widetilde{\mathcal{J}}_G\to \mathcal J_G$, we have an action of $\widetilde{G}$ on the Bass--Serre tree
of the splitting $\mathcal J_G=\pi_1\mathcal T$ from Remark \ref{r:rem}. The kernel of the action on this tree is 
the centre $\Z^{m-n}$ of $\widetilde{\mathcal{J}}_G$. %%mrb2
The resulting graph-of-groups decomposition of $\widetilde{\mathcal{J}}_G$  has in the middle the  vertex group 
$p^{-1}(F_n)\cong F_n\times\Z^{m-n}$, and the other vertex
groups are $p^{-1}(J_i) = J\times \mathbb Z^{m-n}$ (where we write $=$ rather than $\cong$ because the
splitting is unique); the edge groups are of the form
$\mathbb Z^{m-n+1}=\mathbb Z\times \mathbb Z^{m-n}$,
where the edge morphisms respect the $\mathbb Z^{m-n}=H_2(\mathcal J_G,\mathbb Z)$ factor. 
We realise this graph of groups as a graph of spaces with vertex spaces $X_i\times \To^{m-n}$
and $Y_n\times \To^{m-n}$, with the edge-space inclusions represented by the
unique base-point preserving affine maps that induce the desired maps at the level of $\pi_1$.
This inclusion map is the obvious one at the extremal vertices, with the
first factor of $\mathbb Z\times \mathbb Z^{m-n}$ mapping to $\sigma_i(\alpha_i)$, the image of $\alpha_i$
under the unique splitting of $p^{-1}(J_i) \overset{p}\to J_i$. At the middle vertex, the decompositon  $p^{-1}(F_n)\cong F_n\times\Z^{m-n}$
depends on the choice of splitting $\sigma_0:F_n\to p^{-1}(F_n)$ and, correspondingly, while the central
factor $1\times \mathbb Z^{m-n}$ in the edge group maps  to the second factor of $p^{-1}(F_n)$, 
 the generator of the first factor $\Z$ maps to $\sigma_0(r_i)$, which is $\sigma_i(\alpha_i)z_i^{-1}$. %%mrb
 
 Thus we have exhibited $\widetilde{\mathcal{J}}_G$ as the fundamental group of a graph of spaces whose
 total space is $K^{X, c} (\underline{R}, \underline{Z})$ where $\underline{R}=(r_i \mid i=1,\dots,m)$
 and $\underline{Z}=(z_i\mid i=1,\ldots,m)$ with $z_i:=\sigma_0(r_i)^{-1} \sigma_i(\alpha_i)\in \Z^{m-n}$. %%mrb-end
\end{proof}

\begin{example}\label{ex:hig} We shall need groups $J$ with $\wh{J}=1$ that satisfy the hypotheses of Proposition \ref{p: tubular properties}.
The first such group was discovered by  Graham Higman \cite{hig}; it has the
aspherical presentation
$$
J = \<a_1,a_2,a_3,a_4 \mid a_{2}^{-1}a_1a_{2} a_{1}^{-2}, a_{3}^{-1}a_2a_{3} a_{2}^{-2},
a_{4}^{-1}a_3a_{4} a_{3}^{-2}, a_{1}^{-1}a_4a_{1} a_{4}^{-2}\>.
$$
Many more such examples are described in \cite{BG}.  
\end{example}

\section{Proving Theorem \ref{t:tech} with a recursive class of tubular bundles of type $(d;n,m)$}

In  \cite{BWprotrivial} we constructed a recursive sequence of finite group presentations
$\P(k) \equiv \< A \mid R_k\>$ (with $A$ a fixed finite alphabet and $S_k$ of a fixed
cardinality) so that for the groups
$G(k)=|\P(k)|$,
$$
\{k\in\N \mid \widehat{G(k)}\ncong 1\}
$$ is recursively enumerable but not recursive. We may also assume that the groups are perfect.  Let $n=|A|$, let $m=|R_k|$, and let $d=m-n$.

We fix a group $J$ as in Example \ref{ex:hig}, take $\alpha$ to be a generator of infinite order and $X$
to be the standard 2-complex of an aspherical presentation of $J$.
We then apply  the construction of Subsection \ref{ss:main} to the sequence of presentations $\P(k)$
to obtain a recursive sequence of presentations $\J(\P(k))$ for groups $\J_{G(k)}$. 

According to Proposition \ref{p:properties}, the groups $\J_{G(k)}$ are perfect,  
 the presentations  $\J(\P(k))$ are aspherical, and $\wh{\J}_{G(k)}\cong 1$ if and only
 if $\widehat{G(k)}\cong 1$. So Theorem \ref{t:tech} 
would be proved at this stage if each $H_2(\J_{G(k)},\Z)=0$. But euler characteristic tells us that this is not the case,
because $m>n$. We therefore pass to the universal central extension of $\J_{G(k)}$, appealing
to Proposition \ref{p:tilde} to derive a recursive sequence of finite presentations
$\widetilde{\P}_k$ for the groups $\widetilde{\J}_{G(k)}$,  noting that $\widetilde{\J}_{G(k)}$
has a non-trivial finite quotient if and only if ${\J}_{G(k)}$ does.

We will be done if we can descibe an algorithm that, taking each $\widetilde{\P}_k$  in turn, produces an explicit model for the classifying space
$K(\widetilde{\P}_k,1)$.

Proposition \ref{p: tubular properties} assures us that $\widetilde{\J}_{G(k)}$ is the fundamental
group of an aspherical
tubular bundle $K^{X, c} (\underline{R_k}, \underline{Z})$ of type $(d;n,m)$.  %%mrb -- tilde removed
An obvious algorithm enumerates all
such bundles as $\underline{Z}$ varies, and the Seifert--van Kampen theorem
provides an explicit presentation for the fundamental group of each.
Proceeding along this list of presentations with a naive search, we will eventually identify 
an isomorphism between $\widetilde{\J}_{G(k)}$, as presented by $\widetilde{\P}_k$, and 
a presentation on the list. When such is found, our algorithm outputs 
the corresponding complex $K^{X, c} (\underline{R_k}, \underline{Z})$ as %%mrb -- tilde removed
the sought-after model for the classifying space $K(\widetilde{\P}_k,1)$.

This completes the proof of Theorem \ref{t:tech}.
\qed

\begin{remarks}%mrb
(1) The interested reader can verify that the presentations that we constructed to
prove Theorem A have the additional property that the sets $A_n, B_n$ and $S_n$ are of fixed cardinality, while $|R_n|$ is uniformly bounded.

(2)
If there exists a hyperbolic group that is not residually finite, then one can prove Theorem A much more directly: if such a group exists, then Theorem 9.6 of \cite{BWprotrivial} will produce a recursive sequence of finite presentations $\mathcal{Q}_n$, each presenting a perfect hyperbolic group $Q_n$, so that there is no algorithm that can determine which of the groups presented has trivial profinite completion; one can then apply 
the algorithm of  \cite{BReeves}  to construct an explicit model for the 3-skeleton of a classifying space 
for the universal central extension $\widetilde{Q}_n$; and one can use this sequence of complexes in place of those yielded by Theorem B.
\end{remarks}

\end{document}